\documentclass{amsart}

\usepackage{amssymb}
\usepackage[all]{xy}

\numberwithin{equation}{section}

\def\Z{{\mathbb Z}}
\def\Q{{\mathbb Q}}
\def\C{{\mathbb C}}

\def\H{{\mathbb H}}
\def\L{{\mathbb L}}
\def\V{{\mathbb V}}
\def\UU{{\mathbb U}}
\def\MM{{\mathbb M}}

\def\A{{\mathcal A}}

\def\cC{{\mathcal C}}
\def\D{{\mathcal D}}

\def\cG{{\mathcal G}}
\def\cH{{\mathcal H}}

\def\M{{\mathcal M}}

\def\T{{\mathcal T}}
\def\U{{\mathcal U}}

\def\X{{\mathcal X}}
\def\Y{{\mathcal Y}}

\def\G{\Gamma}

\def\a{{\mathfrak a}}

\def\f{{\mathfrak f}}
\def\g{{\mathfrak g}}
\def\h{{\mathfrak h}}

\def\n{{\mathfrak n}}

\def\t{{\mathfrak t}}
\def\u{{\mathfrak u}}

\def\Gm{{\mathbb{G}_m}}
\def\Sp{{\mathrm{Sp}}}
\def\SL{{\mathrm{SL}}}
\def\GL{{\mathrm{GL}}}

\def\orb{\mathrm{orb}}

\def\Atwobar{\overline{\A}_2}
\def\AAbar{\overline{\A}_1\times\overline{\A}_1}

\newcommand\id{\operatorname{id}}

\newcommand\Hom{\operatorname{Hom}}

\newcommand\Aut{\operatorname{Aut}}

\newcommand\Gr{\operatorname{Gr}}

\newtheorem{theorem}{Theorem}[section]

\newtheorem{proposition}[theorem]{Proposition}
\newtheorem{corollary}[theorem]{Corollary}
\newtheorem{bigtheorem}{Theorem}

\theoremstyle{definition}

\newtheorem{example}[theorem]{Example}

\theoremstyle{remark}
\newtheorem{remark}[theorem]{Remark}

\begin{document}


\title{On the completion of the mapping class group of genus two}


\author{Tatsunari Watanabe}
\address{Department of Mathematics, Purdue University, 
West Lafayette}
\email{twatana@purdue.edu}
\urladdr{www.math.purdue.edu/people/bio/twatana} 





%


 \maketitle



\begin{abstract} In this paper, we will study the Lie algebra of the prounipotent radical of the relative completion of the mapping class group of genus two.  In particular, we will partially determine a minimal presentation of the Lie algebra by determining the generators and bounding the degree of the relations of it. 
	
\end{abstract}

\section{Introduction}
Let $X$ be a complex smooth algebraic variety. Fix a point $x$ in $X$. The Lie algebra $\g$ of the unipotent (Malcev) completion over $\Q$ of $\pi_1(X,x)$ carries a natural $\Q$-MHS. A minimal presentation of $\Gr^W_\bullet\g$ was studied by Morgan in \cite{mor}. It is generated in weight $-1$ and $-2$ by $H_1(X, \Q)$ and has relations in weight $-2,-3$, and $-4$ that come from $H_2(X, \Q)$.  \\
\indent Denote by $\M_{g}$ the moduli stack of projective smooth curves of genus $g$ over $\C$. Assume that $g\geq 2$. It will be considered as an orbifold in this paper.  It is a natural problem to look for a presentation of the Lie algebra of the unipotent completion of $\pi_1(\M_g)$, since $\M_g$ is smooth and  has a finite cover by a smooth algebraic variety. The issue is that the rational homology $H_1(\M_g, \Q)$ vanishes and hence the unipotent completion is trivial. Instead, we will consider the relative completion of a discrete group due to Deligne, which generalizes the unipotent completion.   Associated to the universal family $f:\cC_g\to \M_{g}$, there is a natural monodromy representation in the orbifold sense
$$\rho_x: \pi_1^\orb(\M_{g},x)\to \Sp(H^1(C,\Q)),$$
where $C$ is the fiber of $f$ over $x$. The orbifold fundamental group $\pi_1(\M_g, x)$ is naturally isomorphic to the mapping class group $\G_g$. It is the group of isotopy classes of orientation-preserving diffeomorphisms of a compact oriented surface $S_g$ of genus $g$.  The representation $\rho_x$ can be also seen from the action of $\G_{g}$ on $H_1(S_g, \Z)$. The relative completion $\cG_g$ of $\G_g$ with respect to $\rho_x$ is the inverse limit of all linear algebraic $\Q$-group $G$ that is an extension of $\Sp_g$ by a unipotent $\Q$-group and such that there is a Zariski-dense representation $\phi_G:\G_g\to G(\Q)$ making the diagram
$$\xymatrix{
	\G_g\ar[rd]^{\rho_x}\ar[d]^{\phi_G}&\\
	G(\Q)\ar[r]&\Sp_g(\Q)
	}
$$
commute.
The completion $\cG_g$ is a proalgebraic $\Q$-group that is an extension of $\Sp_g$ with a prounipotent $\Q$-group $\U_g$.
  In \cite{hain3}, Hain further developed the theory of relative completion and constructed a canonical MHS on the Lie algebra $\u_g$ of $\U_g$.  The associated graded Lie algebra $\Gr^W_\bullet\u_g$ is the one we are interested in.  Hain proved in \cite{hain4, hain0} that for $g\geq 4$, the Lie algebra $\u_g$ is generated in weight $-1$ and quadratically presented and that for $g=3$, it is generated in weight $-1$ and admits quadratic and cubic relations. \\
\indent  Hain used the Johnson's fundamental work \cite{joh1} to determine the generators  and Kabanov's theorem \cite{kab} on $H^2(\M_g, \V)$, where $\V$ is a symplectic local system, to bound the degree of relations, which then allowed him to explicitly determine the relations by producing certain commuting two elements in the Torelli group. In this paper, we will partially extend Hain's work to the case $g=2$. More precisely, we will determine the generator of the graded Lie algebra $\Gr^W_\bullet \u_2$ and bound the degree of its relations. Let $V_{a+b}$ denote the irreducible representation of $\Sp_2$ corresponding to the partition $n=a+b$ with $a\geq b$ of a nonnegative integer $n$. By default, we assume that $V_{a+b}$ has weight $a+b$ and let $V_{a+b}(m)$  denote the Tate twist of $V_{a+b}$ by $\Q(m)$ of weight $a+b-2m$.
\begin{bigtheorem}\label{main result} There is an isomorphism of graded Lie algebras 
	$$
	\L(V_{2+2}(3))/ R\cong \Gr^W_\bullet\u_2
	$$
	in the category of $\Sp_2$-representations. The Lie ideal $R$ is generated by a subrepresentation of
	$$\bigoplus_{a+b}V_{a+b}(a+2)$$ 
	where the sum  is taken over the partitions $a+b$ of $2,4,6,10$ with $a>b$. In particular, $\Gr^W_\bullet\u_2$ is finitely presented with possibly the  relations of cubic to septic degree. 
\end{bigtheorem}
 This will follow immediately from Corollary \ref{generator of completion} and Corollary \ref{relations}.\\
 \indent Our approach is fundamentally different from the higher genus case by Hain. Our main result is largely due to Petersen's work \cite{pet2} on the cohomology of symplectic local systems on $\A_2$. Together with Petersen's nonvanishing theorem \cite[Thm.~5.1]{pet3} and an extension of it, we are able to determine the generators and bound the degree of relations. While the computations for $g\geq 3$ are topological in nature, our approach is arithmetic.  
 This is because, $\M_2$ being the complement of the locus $\D_{1,1}$ of the products of two elliptic curves in $\A_2$,  the cohomology of $\M_2$ is determined by that of $\D_{1,1}$ and $\A_2$ by the Gysin sequence.\\
 \indent In sections 2,3, and 4, we will review briefly the mapping class groups, the moduli stacks of smooth projective curves and principally polarized abelian varieties, and set notations for our basic objects for the case $g=2$. In section 5, we will review the theory of relative completion along its Hodge theory aspect developed by Hain and then discuss  a minimal presentation of $\Gr^W_\bullet\u_g$ for $g\geq3$. In section 6, we will determine the generators of $\Gr^W_\bullet\u_2$ and in section 7, we will prove a nonvanishing theorem for a Gysin sequence to bound the degree of relations. 
 
 {\em Acknowledgements:} I am very grateful to Dan Petersen for taking an interest in this work, answering many questions, and sharing his unpublished notes on the cohomology of the mapping class group of genus two. I am also very grateful to Richard Hain who introduced to me the study of the mapping class groups and moduli of curves. I also would like to thank Kevin Kordek for numerous meaningful discussions on this work. 

\section{Mapping class groups and Torelli groups}
Let $S_g$ be a compact oriented surface of genus $g$. Let $P$ be a subset of $S_g$ consisting of $n$ distinct points on $S_g$. The group of the orientation-preserving diffeomorphisms of $S_g$ fixing $P$ pointwise will be denoted by $\mathrm{Diff}^+(S_g, P)$, which is given open-compact topology.  The mapping class group of type $(g,n)$, dentoed by $\G_{g,n}$, is defined to be the group of path-connected components of $\mathrm{Diff}^+(S_g, P)$:
$$\G_{g,n}=\pi_0(\mathrm{Diff}^+(S_g, P)).$$
We easily see that the group $\G_{g,n}$ is independent of the choice of the reference surface $S_g$ and $P$. In this paper, we will always assume that $2g-2+n>0$. When $n=0$, $\G_{g,0}$ will be denoted by $\G_g$. \\
\indent The first integral cohomology group $H^1(S_g, \Z)$ is equipped with the cup product pairing   
$$\langle~,~ \rangle: H^1(S_g, \Z)\otimes H^1(S_g, \Z)\to H^2(S_g,\Z)\cong\Z,$$
that is a nondegenerate, skew symmetric bilinear paring. The mapping class group $\G_{g,n}$ acts on the first homology group $H_1(S_g, \Z)$ and hence on $H^1(S_g,\Z)$,  preserving the pairing $\langle~,~\rangle$.  Therefore, there is a natural homomorphism 
$$\rho:\G_{g,n}\to \Aut(H^1(S_g,\Z), \langle~,~\rangle)=\Sp(H^1(S_g,\Z))=\Sp(H_1(S_g,\Z)).$$
It is well-known that if $g\geq 1$, then the homomorphism $\rho$ is surjective. The kernel of $\rho$ is called the  Torelli group, denoted by $T_{g,n}$.  The group $T_1$ is trivial, Mess showed in \cite{mes} that $T_2$ is an infinitely generated free group , and Johnson proved in \cite{joh} that $T_g$ is finitely generated for $g\geq 3$. 
It is not known whether $T_g$ is finitely presented or not for $g\geq 3$. Hain showed in \cite{hain0} that the Lie algebra $\t_g$ of the Malcev completion of $T_g$ is finitely presented.

\section{Moduli stacks $\M_{g,n}$ and $\A_g$}
By a complex curve $C$ of genus $g$, we mean a Riemann surface of genus $g$, and a marking is an orientation-preserving diffeomorphism $f:S\to C$ to a compact complex curve.  The Teichm\"uller space of marked $n$-pointed compact curves of genus $g$ will be denoted by $\X_{g,n}$. As a set , it consists of the isotopy classes of the markings. The mapping class group $\G_{g,n}$ acts on $\X_{g,n}$ by $[\phi]:[f]\mapsto [f\circ \phi^{-1}]$. The fact that this action is properly discontinuous and virtually free imply that the quotient 
$$\M_{g,n}=\X_{g,n}/\G_{g,n}$$
 is a complex orbifold and that it is covered by smooth complex varieties. \\
 \indent Let $\h_g$ be the Siegel upper half-space. It consists of the complex $g\times g$ symmetric matrices $Z$ such that $\mathrm{Im} (Z)$ is positive definite. It is the moduli space of principally polarized abelian varieties equipped with a  choice of a symplectic basis and comes with an $\Sp_g(\Z)$-action:
 $$\alpha\cdot Z= \left[\begin{array}{cc}A& B\\C&D\end{array}\right]\cdot Z=(AZ+B)(CZ+D)^{-1}.$$
 The moduli stack $\A_g$ of principally polarized abelian varieties  is the quotient 
 $$\A_g=\h_g/\Sp_g(\Z).$$
It is considered as a complex orbifold. In Section \ref{bound}, we will also view it as a locally symmetric space for $\Sp_g(\mathbb{R})$ and consider its Baily-Borel compactification and the boundary cohomology  at the zero-dimensional stratum. \\
\indent The Torelli map $\M_g\to \A_g$ as stacks are a 2-to-1 map ramified along the hyperelliptic locus, $\cH_g$, for $g\geq3$ and it is an open immersion for $g=2$. For $g=1$, the Torelli map $\M_{1,1}\to\A_1$ is an isomorphism of stacks. In this paper, we especially pay our attention to the case $g=2$. Via the Torelli map $i:\M_2\to \A_2$, we may consider $\M_2$ as an open substack of $\A_2$. The complement in $\A_2$ of $\M_2$ is a smooth divisor consisting of the unordered products of two elliptic curves, which we will denote by $\D_{1,1}$.

\section{Local systems over $\M_2$, $\A_2$, and $\D_{1,1}$}
Let $f:\Y\to \A_2$ be the universal family over $\A_2$.  Let $\V$ be  the sheaf $\mathrm{R}^1f_\ast\Q$  over $\A_2$.  That $f$ is smooth and proper implies that $\V$ is a polarized variation of Hodge structure of weight $1$. Let $x$ be in $\A_2$ and $A$ be the fiber over $x$.   The monodromy action of $\pi_1^\orb(\A_2,x)$  on $H^1(A, \Q)$ corresponds to the inclusion 
$$\rho_x:\pi_1^\orb(\A_2,x)=\Sp_2(\Z)\to \Sp_2(\Q).$$
 A finite dimensional $\Sp_2(\Q)$-representation $M$ defines a local system $\MM$ over $\A_2$ via the homomorphism $\rho_x$. In fact, the local system $\MM$ underlies a variation of Hodge structure. \\
\indent We will briefly review the $\Sp_2$-representation theory. We consider $\Sp_2$ as a linear algebraic group over $\Q$. \indent The maximal torus $T$ of $\Sp_2$ consists of the diagonal matrices $t$ of the form $\mathrm{diag}(t_1,t_2,t_1^{-1},t_2^{-1})$ and the fundamental dominant weights $\lambda_j(t):T\to \mathbb{R}$ are given by 
$$\lambda_1(t)=t_1\text{ and }\lambda_2(t)=t_1t_2.$$
Then there is a bijection between the positive integral linear combinations $\alpha=n_1\lambda_1+n_2\lambda_2$  and the set of isomorphism classes of irreducible $\Sp_2$-representations $M_{[n_1,n_2]}$ with highest weight $\alpha$.\\
\indent We may also produce irreducible representations by Weyl's construction (see \cite{FH}).
Let $V$ be a 4-dimensional vector space over $\Q$ equipped with a nondegenerate, skew symmetric bilinear form $\V\otimes\V\to \Q$.  For a partition $\lambda: n=a+b$ of an nonnegative integer $n$ with $a\geq b\geq 0$, there is an irreducible representation of $\Sp_2$, which we denote by $V_{a+b}$. The representation $V_{a+b}$ can be realized as a summand in $V^{\otimes n}$ which is specified by contraction maps and the Schur functor corresponding to the partition $\lambda$. The irreducible representation $V_{a+b}$ corresponds to $M_{[a-b,b]}$. In this paper, we will use the partition notation. \\ 
\indent Note that there is an isomorphism of local systems $\V\cong \V_1$.  Moreover, we obtain the local system $\V_{a+b}$ corresponding to the irreducible representation $V_{a+b}$. As $\V$ is a variation of Hodge structure of weight $1$,  Weyl's construction implies that $\V_{a+b}$ carries a variation of Hodge structure of weight $n=a+b$. By pulling back, we also obtain variations of Hodge structures over $\M_2$ and $\D_{1,1}$, which we will also denote by $\V_{a+b}$. Therefore, the cohomology  $H^\bullet(X, \V_{a+b})$ with $X\in\{\M_2, \A_2, \D_{1,1}\}$ carry natural Mixed Hodge structures, and so do the compactly supported cohomology $H^\bullet_c(X, \V_{a+b})$. We will abbreviate Hodge structure, Mixed Hodge structure, and variations of them by HS, MHS, VHS, and VMHS, respectively.\\

We will need to understand the restrictions of $\V_{a+b}$ to $\D_{1,1}$. As an orbifold, $\D_{1,1}$ is the quotient 
$$\D_{1,1}= \h_1\times\h_1/(\SL_2(\Z)\times\SL_2(\Z)\rtimes S_2),$$ 
where $S_2$ interchanges the components of $\SL_2\times \SL_2$. The space $\h_1\times\h_1$ is considered to be the subspace of $\h_2$:
$$\h_1\times\h_1=\left\{\begin{bmatrix}
	z_1&0\\
	0&z_2\\
	\end{bmatrix}| \mathrm{Im}(z_i)>0, i=1,2\right\}.$$
The wreath product $\SL_2(\Z)\times\SL_2(\Z)\rtimes S_2$ embeds into $\Sp_2(\Z)$, consisting of the matrices 
$$\begin{bmatrix}
a_1&0&b_1&0\\0&a_2&0&b_2\\c_1&0&d_1&0\\0&c_2&0&d_2
\end{bmatrix}, \hspace{.5in} \begin{bmatrix}
0&a_1&0&b_1\\a_2&0&b_2&0\\0&c_1&0&d_1\\c_2&0&d_2&0
\end{bmatrix}, \hspace{.5in} \begin{bmatrix}
a_i&b_i\\c_i&d_i
\end{bmatrix}\in \SL_2(\Z)
$$

Thus the restriction of the local system $\V_{a+b}$ to $\D_{1,1}$ corresponds to the restriction of the $\Sp_2$-representation $V_{a+b}$ to the subgroup $\SL_2(\Z)\times\SL_2(\Z)\rtimes S_2$ of $\Sp_2$. So we need to know how $V_{a+b}$ decomposes as the sum of the irreducible  $\SL_2\times\SL_2\rtimes S_2$-representations. Petersen gives  a branching formula in \cite[Prop.3.4]{pet1} for this inclusion. \\
\indent Denote the standard representation of $\SL_2$ by $H$. The symmetric tensor power $\mathrm{Sym}^mH$ will be denoted by $H_m$. These are the irreducible representations of $\SL_2$. The irreducible representations of $\SL_2\times\SL_2\rtimes S_2$ are given by the following three types:
\begin{enumerate}
	\item $U_{a,b}=H_a\otimes H_b\oplus H_b\otimes H_a$, $\sigma: u\otimes v+ v'\otimes u'\mapsto u'\otimes v'+ v\otimes u$, with $a\not=b$,
	\item $U_a^+=H_a\otimes H_a$, $\sigma: u\otimes v\mapsto v\otimes u$,
	\item $U_a^-=H_a\otimes H_a$, $\sigma: u\otimes v\mapsto -v\otimes u$,
\end{enumerate}
where $S_2=\langle\sigma\rangle$. Petersen's branching formula explicitly describes the restriction $\mathrm{Res}_{\SL_2\times\SL_2\rtimes S_2}^{\Sp_2}V_{a+b}$ as the sum of the irreducible representations of the above three types. 

\begin{example} The key representations that we will consider  are $V_{2l+2l}$ with $l\geq 1$ and $V_{a+b}$ with $a+b$ even and $a>b$. The restriction $\mathrm{Res}_{\SL_2\times\SL_2\rtimes S_2}^{\Sp_2}V_{2l+2l}$ contains a single copy of the trivial representation $U^+_0$ as a summand and $\mathrm{Res}_{\SL_2\times\SL_2\rtimes S_2}^{\Sp_2}V_{a+b}$ contains a single copy of $U_{a-b,0}$. The corresponding local systems $\UU_{a,b}$, $\UU^+_a$, and $\UU^-_a$ over $\D_{1,1}$ also carry VHSs, and in order to be consistent with Hodge weights, we need to apply Tate twists. That is,  $U^+_0$ is twisted by $-2l$, i.e., $U^+_0(-2l)$ and $U_{a-b,0}$ by $-b$, $U_{a-b,0}(-b)$. 

\end{example}

As an orbifold, the moduli stack $\M_{1,1}$ is the quotient $\M_{1,1}=\h_1/\SL_2(\Z)$ with $\pi_1^\orb(\M_{1,1})=\SL_2(\Z)$. Each representation $H_m$ defines a local system $\H_m$ over $\M_{1,1}$ underlying a VHS of weight $m$. Recall that the Torelli map $\M_{1,1}\to \A_1$ is an isomorphism of stacks.  The cohomology of the local systems over $\D_{1,1}$ can be computed by pulling back them along the map $\A_1\times\A_1\to \D_{1,1}$.
\begin{remark}\label{cohomology on D11}  There are isomorphisms
	\begin{enumerate}
		\item $H^\bullet(\D_{1,1}, \UU_{a,b})\cong H^\bullet(\M_{1,1}, \H_a)\otimes H^\bullet(\M_{1,1}, \H_b)$,
		\item $H^\bullet(\D_{1,1}, \UU_{a}^+)\cong \mathrm{Sym}^2H^\bullet(\M_{1,1}, \H_a)$, and
		\item $H^\bullet(\D_{1,1}, \UU^-_a)\cong \Lambda^2 H^\bullet(\M_{1,1}, \H_a)$.
		
	\end{enumerate}
	
\end{remark}

\section{Review of relative completion of $\G_{g,n}$ and a minimal presentation of the Lie algebra $\Gr^W_\bullet\u_g$}
 A detailed treatment of the theory of relative completion can be found in \cite{hain3} and \cite{hain2}.  Here, we will summarize and state the properties of the theory needed for our main results.\\
 \indent 
Assume that $F$ is a field of characteristic zero. By an algebraic $F$-group, we mean an affine group scheme $G$ of finite type over $F$. A proalgebraic $F$-group is the projective limit of algebraic $F$-groups. \\
\indent Consider the following data
\begin{enumerate}
	\item $R$ is a reductive group over $F$, and 
	\item $\G$ is a discrete group with a Zariski-dense homomorphism 
	$\rho:\G\to R(F).$
\end{enumerate}
The relative completion of $\G$ with respect to $\rho$ is a pair ($\cG, \tilde{\rho}$) of a proalgebraic $F$-group $\cG$  that is an extension of $R$ by a prounipotent $F$-group $\U$ and a natural map $\tilde{\rho}:\G\to \cG(F)$, satisfying the universal property:
If $G$ is a proalgebraic $F$-group that is also an extension of $R$ by a prounipotent $F$-group $U$ such that $\rho$ factors through $G(F)\to R(F)$, then there exists a unique morphism of proalgebraic $F$-groups $\phi: \cG\to G$ such that the diagram
$$\xymatrix{
	  \G\ar[r]^{\tilde{\rho}}\ar[d]&\cG(F)\ar[d]\ar[ld]_\phi\\
	  G(F)\ar[r]&R(F)
}
$$ commutes. 
One of the key properties of relative completion we will need is the following. Denote the Lie algebra of $\U$ by $\u$.  For our purpose, we take $R=\Sp_g$ and $F=\Q$, and the cohomology we use here is the continuous cohomology of a discrete group and pro-Lie algebras. 
\begin{proposition}\label{key property} For each partition $\lambda$, let $V_\lambda$ be the corresponding irreducible representation of $\Sp_g$. Then
	\begin{enumerate}
		\item there is a natural $\Sp_g$-invariant isomorphism
		$$H^1(\u)\cong \bigoplus_{\lambda}H^1(\G, V_{\lambda})\otimes V_{\lambda}^\ast,$$
		and 
		\item there is a natural $\Sp_g$-invariant injection
		$$H^2(\u)\to \bigoplus_{\lambda}H^2(\G, V_{\lambda})\otimes V_{\lambda}^\ast.$$
	\end{enumerate}
\end{proposition}

We take the relative completion of $\G_{g,n}$ with respect to the homomorphism $\rho: \G_{g,n}\to \Sp(H^1(S_g, \Q))$. Since the image of $\rho$ is $\Sp(H^1(S_g, \Z))$, $\rho$ has a Zariski-dense image. Denote the completion of $\G_{g,n}$ by $\cG_{g,n}$ and its prounipotent radical by $\U_{g,n}$. The Lie algebras of $\cG_{g,n}$ and $\U_{g,n}$ will be denoted by $\g_{g,n}$ and $\u_{g,n}$. \\
\indent For each $x$ in $\M_{g,n}$,  there is  a monodromy representation 
$$\rho_x:\pi_1^\orb(\M_{g,n}, x)\to \Sp(H^1(C, \Q)),$$
where $C$ is the fiber of the universal family over the point $x$. The natural isomorphism $\G_{g,n}\cong \pi_1^\orb(\M_{g,n},x)$ identifies this monodromy action with the homomorphism $\rho$. Denote the completion of $\rho_x$ by $\cG_{g,n}(x)$ and its prounipotent radical by $\U_{g,n}(x)$. Their Lie algebras are denoted by $\g_{g,n}(x)$ and $\u_{g,n}(x)$, respectively. The following result in \cite{hain3} implies that for each choice of $x$, the Lie algebras $\g_{g,n}(x)$ and $\u_{g,n}(x)$ carry canonical $\mathbb{R}$-MHSs.
\begin{theorem}[Hain]\label{hs on relative completion}
	Suppose that $X$ is a smooth complex variety and that $\V$ is a polarized variation of $\Q$-HS over $X$ of geometric origin. If the monodromy action 
	$$\rho_x:\pi_1(X,x)\to \Aut(\V_x, \langle~,~\rangle)$$ has a Zariski-dense image, then the Lie algebra $\g$ of the completion of $\pi_1(X,x)$ and the pronilpotent Lie algebra $\u$ of the prounipotent radical of the completion admit natural $\Q$- MHSs, where brackets are morphisms of MHSs.
\end{theorem}
\begin{remark}
	The original statement is stated over $\mathbb{R}$. However, if the Zariski-closure of the image of $\rho_x$ is defined over $\Q$, the $\mathbb{R}$-structure canonically lifts to a $\Q$-structure. For a concrete explanation, see  \cite[4.3]{hain5}. 
\end{remark}
Fix a base point $x$ in $\M_{g,n}$. We will omit the reference to the base point $x$ from now on. Now, we review the presentations of the pronilpotent Lie algebras $\u_{g}$ for $g\geq 3$. We will need the following standard fact.
\begin{proposition} Let $ \n$ be a negatively graded Lie algebra. Let $\f$ be the free Lie algebra generated by $H_1(\n)$, i.e., $\f=\L(H_1(\n))$. Then there is an injection of graded Lie algebras $\gamma: H_2(\n)\to [\f,\f]$ such that the composition $H_2(\n)\to [\f,\f]\to\Lambda^2H_1(\n)$ is the dual of the cup product and 
	$$\n\cong \f/\langle \mathrm{im}\gamma\rangle.$$
	\qed
\end{proposition}

The Lie algebra $\u_{g}$ is pronilpotent and hence graded by bracket. It is also graded by the weight filtration. For $g\geq 3$, it follows from the Johnson's work \cite{joh1} that $H_1(\u_{g})$ is a $\Q$-HS of weight $-1$, which together with the strictness of morphisms of MHSs implies that the weight filtration agrees with its lower central series (for a proof, see \cite[Lem. 4.7]{hain0}). The induced weight filtration on $\L(H_1(\u_{g}))$ is also the lower central series. The associated graded Lie algebra $\Gr^W_\bullet\u_{g}$ has a minimal presentation
$$\Gr^W_\bullet\u_{g}\cong \L(H_1(\u_{g}))/\langle \Gr^W_\bullet H_2(\u_{g})\rangle$$
as a graded Lie algebra in the category of the $\Sp_g$-representations. 
The key points for determining this presentation are the Johnson's work on the abelianization of the Torelli group $T_{g,1}$ for determining $H_1(\u_g)$ and Kabanov's purity theorem \cite{kab} for bounding $H_2(\u_g)$:
\begin{theorem}[Kabanov] Let $V_\lambda$ be an irreducible rational representation of $\Sp_g$ and $\V_\lambda$ be the corresponding VHS on $\M_g$.  Then the weights of the MHS $H^2(\M_g, \V_\lambda)$ lie in $\{2+|\lambda|,3+|\lambda|\}$ for $3\leq g <6$ and the weight is equal to $2+|\lambda|$ for $g\geq 6$, where $|\lambda|$ denotes the weight of $\V_\lambda$.  
\end{theorem}	
In Proposition \ref{weights of the second cohomology g=2}, we will extend Kabanov's result to the case when $g=2$. 
The relation between the MHSs $H_\bullet(\u_g)$ and $H^\bullet(\M_g, \V)$ for low degrees are given by
\begin{proposition}[{\cite[Prop.7.1 \& 7.3]{hain0}}]Suppose that $g\geq 1$. If $\V$ is the VHS over $\M_{g,n}$ corresponding to an irreducible rational $\Sp_g$-representation $V$, then for each $k$, there is a morphism of MHSs
	$$\Hom_{\Sp_g}(H_k(\u_{g,n}), V)\to H^k(\M_{g,n}, \V),$$
which is an isomorphism for $k=0,1$ and an injection for $k=2$. 
\end{proposition}
\begin{remark}
	In \cite[Prop.~7.3]{hain0}, the result is stated for the case when $g\geq 3$, but since the canonical $\mathbb{R}$-MHS on $\u_{g,n}$ lifts to a $\Q$-MHS, the result can be stated for $g\geq1$ as well. 
\end{remark}
From this, it follows that $\Gr^W_\bullet\u_g$ has possibly quadratic and cubic relations for $3\leq g< 6$, and only quadratic relations for $g\geq 6$. In \cite{hain0}, Hain determined all quadratic relations for $g\geq 4$, and in \cite{hain4} showed that for $g\geq 4$,  $\Gr^W_\bullet\u_g$ is quadratically presented and determined the quadratic and cubic relations for $g=3$. Hain used the fact that $\Gr^W_\bullet\u_3$ is not quadratically presented to show that $\G_3$ does not arise as a fundamental group of a compact K\"ahler manifold.
\begin{theorem}[Hain]
	\begin{enumerate}
		\item If $g\geq 3$, then 
		$$H_1(\u_g)\cong V_{1+1+1}.$$
		\item There are isomorphisms
		$$\Gr^W_{-2}H_2(\u_g)\cong \left\{
		\begin{array}{ll}
		V_0 & g=3 \\
		\text{the $\Sp_g$-complement of } V_{2+2} \text{ in }\L_2(V_{1+1+1})& g\geq 4, \\
		\end{array} 
		\right. 
				$$
				and
			$$\Gr^W_{-3}H_2(\u_g)\cong \left\{
			\begin{array}{ll}
			\text{the $\Sp_g$-complement of } V_{3+1+1} \text{ in }\L_3(V_{1+1+1}) & g=3 \\
			0& g\geq 4, \\
			\end{array} 
			\right. 
						$$
						where $\L_k(V)$ denotes the $k$th component of the free Lie algebra $\L(V)$. 
	\end{enumerate}
	Note that Tate twists are omitted from the right-hand terms.
\end{theorem}

\section{Generators of the Lie algebra $\u_2$}
By Proposition \ref{key property}, the generators  of $\u_2$ are determined by the cohomology
groups $H^1(\M_2, \V_{a+b})$, which we will determine in this section. 

 Recall that the Torelli map $i:\M_2\to \A_2$ is an open immersion and the complement of the image is a divisor, denoted by $\D_{1,1}$, whose generic point is a product of two unordered elliptic curves.  The product $\A_1\times \A_{1}$ is a two-sheeted cover of $\D_{1,1}$, where the symmetric group $S_2$ simply interchanges the two elliptic components. Each irreducible rational representation $V_{a+b}$ of $\Sp_2$ defines a variation of Hodge structure $\V_{a+b}$ of weight $a+b$ on $\A_2$. The involution acts on $\V_{a+b}$ by $(-1)^{a+b}\id$, while it acts trivially on the cohomology, and hence the cohomology of $\V_{a+b}$ vanishes when $a+b=$ odd. Thus we will only need to consider the case when $a+b$ even.  The restriction of $\V_{a+b}$ to $\M_2$ and $\D_{1,1}$ are also denoted by $\V_{a+b}$.   Associated to the open immersion $i:\M_2\to \A_2$, there is a Gysin sequence
 $$  H^1(\A_2, \V_{a+b})\to H^1(\M_2, \V_{a+b})\to H^0(\D_{1,1}, \V_{a+b}(-1))
  \to H^2(\A_2, \V_{a+b}).$$
  
By a result of Borel in \cite{Bor}, for $a+b>0$, the cohomology group $H^1(\A_2, \V_{a+b})\cong H^1(\Sp_2(\Z), V_{a+b})$ vanishes. Thus $H^1(\M_2, \V_{a+b})\cong H^1(\G_2, V_{a+b})$ is the kernel of the Gysin map $H^0(\D_{1,1}, \V_{a+b}(-1))\to H^2(\A_2, \V_{a+b})$. 
It follows immediately from \cite[Thm.~2.1]{pet2} that $H^2(\A_2, \V_{a+b})$ vanishes except when  $a=b$ even and the dimension of the cusp forms for $SL_2(\Z)$ of weight $a+b+4$ is positive, and that when $a=b=2l$ and $l\geq 2$, $H^2(\A_2, \V_{2l+2l})$ is the direct sum of copies of Tate HS $\Q(-2l-1)$, having dimension at least one.   Firstly, we note that the branching formula \cite[Prop.3.4]{pet1} implies that $H^0(\D_{1,1}, \V_{a+b})=0$ unless $a=b$ even. Therefore, we will only need to consider the case when $a=b=2l$. In this case, the restriction of $\V_{2l+2l}$ to $\D_{1,1}$ contains a single copy of the trivial local system $\UU^{+}_0(-2l)$ as a summand and thus $H^0(\D_{1,1}, \V_{2l+2l})(-1)=\Q(-2l-1)$. Secondly, we observe that for $l=1$, $H^2(\A_2, \V_{2+2})=0$, since there is no cusp form of weight 8, and hence $H^1(\G_2, V_{2+2})=\Q(-3)$. For $l\geq 2$, we need the following nonvanishing theorem by Petersen. 

\begin{theorem}[{\cite[Thm.~5.1]{pet3}}]\label{gysin map 1}
	For $l\geq 2$, the Gysin map $$H^0(\D_{1,1}, \V_{2l+2l})(-1)\to H^2(\A_2, \V_{2l+2l})$$ is nontrivial. 
\end{theorem}
\begin{remark}
	We will adopt Petersen's proof of this theorem  to bound $H_2(\u_2)$. The strategy and more detail will be given in section \ref{bound}.
\end{remark}
 Therefore, we have
\begin{corollary}\label{generator}
	There is an isomorphism of $\Q$-HSs
	$$H^1(\M_2, \V_{a+b})\cong \left\{
	\begin{array}{ll}
	\Q(-3) & \text{ if }a=b=2 \\
	 0& \text{otherwise}. \\
	\end{array} 
	\right. 
	$$ \qed
\end{corollary}
Proposition \ref{key property} gives us
\begin{corollary}\label{generator of completion} There is an isomorphism of $\Q$-HSs
	$$H_1(\u_2)\cong V_{2+2}(3).$$ 
\end{corollary}
%
%
%
%
%
%
%

That $H_1(\u_2)$ is pure of weight $-2$ and that the bracket is a morphism of MHSs give 

\begin{corollary}\label{filtration} For $m\geq 1$, we have
	$$W_{-2m+1}\u_2=W_{-2m}\u_2=L^m\u_2,$$ 
	where $L^\bullet\n$ denotes the lower central series of $\n$. 
	\qed
\end{corollary}

\section{Bounds for  the degree of the relations of the Lie algebra $\u_2$}\label{bound}

In this section, we will bound the cohomology $H^2(\u_2)$ by partially computing $H^2(\M_2, \V_{a+b})$. This will be done by considering the natural inclusion 
$$(H^2(\u_2)\otimes V_{a+b})^{\Sp}\to H^2(\M_2, \V_{a+b})$$ that is a morphism of MHS.
We will consider the part of the Gysin sequence
$$H^0(\D_{1,1}, \V_{a+b}(-1))\to H^2(\A_2, \V_{a+b})\to H^2(\M_2, \V_{a+b})\hspace{1in}$$
$$\hspace{1in} \to H^1(\D_{1,1}, \V_{a+b}(-1))\to H^3(\A_2, \V_{a+b}).$$ 

The following result is a key to bound $H^2(\u_2)$ and follows easily from Petersen's work \cite{pet2}.
\begin{proposition} \label{weights of the second cohomology g=2}
	\begin{enumerate} Assume that $a+b$ is even.
		\item $H^2(\M_2, \V_{a+b})=0$ if $a=b$ odd.
		\item The possible weights of $H^2(\M_2, \V_{a+b})$ are given by
		$$\left\{
		\begin{array}{ll}
		2a+2 &\text{ if } a=b \text{ even } \\
		a+b+3,\hspace{.1in}2a+4=a+b+a-b+4	& \text{ if }a>b. \\
		\end{array} 
		\right. 
		$$
	\end{enumerate}
\end{proposition}
\begin{proof}
	For (i), first we note that the cohomology group $H^2(\A_2, \V_{a+a})$ vanishes \cite[Thm.~2.1]{pet2} and that Petersen's branching formula  implies that $H^1(\D_{1,1}, \V_{a+a})$ vanishes. Thus the result follows.\\
	\indent For (ii), first assume that $a=b$ even. Then similarly, we have $H^1(\D_{1,1}, V_{a+a})=0$. Thus the restriction map $H^2(\A_2, V_{a+a})\to H^2(\M_2, \V_{a+a})$ is surjective. By \cite[Thm.~2.1]{pet2}, $H^2(\A_2, \V_{a+a})$ is pure of weight $2a+2$ and hence the claim. Next, assume that $a>b$. The branching formula \cite[Prop.3.4]{pet1} by Petersen shows that the restriction of $\V_{a+b}$ to $\D_{1,1}$ contains a single copy of the local system $\UU_{a-b,0}(-b)$ as a summand and 
	 that it is the only term contributing to cohomology, i.e., we have $H^1(\D_{1,1}, \V_{a+b}(-1))=H^1(\D_{1,1}, \UU_{a-b,0})(-b-1)$. From Remark \ref{cohomology on D11} 
    We have  isomorphisms
    $$H^1(\D_{1,1}, \UU_{a-b,0})=H^1(\A_1, \H_{a-b})\otimes H^0(\A_1, \Q)\cong H^1(\SL_2(\Z), H_{a-b}).$$ The weight filtration on $H^1(\SL_2(\Z), H_{a-b})$ splits (\cite{ser}) and hence we have 
    $$H^1(\D_{1,1}, \UU_{a-b,0})(-b-1)=(W_{a-b+1}H^1(\SL_2(\Z), H_{a-b}))(-b-1)\oplus \Q(-a-2).$$ 
	The first summand is a $\Q$-HS of weight $a+b+3$ and the second a Tate HS of weight $2a+4$. Since $H^2(\A_2, \V_{a+b})=0$ in this case, our claim follows. 
\end{proof}

Recall that when $\n$ is a pronilpotent Lie algebra that is negatively weighted, there is an injection of graded vector spaces $r:H_2(\n)\to \f:=\L(H_1(\n))$ such that the composition $H_2(\n)\overset{r}\to [\f, \f]\to \Lambda^2 H_1(\n)$ is the dual to the cup product $\Lambda^2 H^1(\n)\to H^2(\n)$.
 By Proposition \ref{filtration}, $H_1(\u_2)$ is pure of weight -2, and hence the  free Lie algebra $\L(H_1(\u_2))$ has only even weights. Then this implies that $H_2(\u_2)$ has only even weights $-2m$ with $m\geq 2$. Therefore, we will only need to consider the HS $\Q(-a-2)$ appearing as a suumand in $H^1(\D_{1,1}, \V_{a+b}(-1))$ with $a>b$. \\
\indent In order to determine whether the Tate term $\Q(-a-2)$ appears in $H^2(\M_2, \V_{a+b})$, we will consider the Gysin map 
$$\delta: H^1(\D_{1,1}, \V_{a+b}(-1))\to H^3(\A_2, \V_{a+b}).$$
 We note from \cite[Thm.~2.1]{pet2} that  $H^3(\A_2, \V_{a+b})$ contains the direct sum of $s_{a+b+4}$ copies of $\Q(-a-2)$ as a summand, where $s_m$ denotes the dimension of the cusp forms for $SL_2(\Z)$ of weight $m$. Petersen kindly shared with me his unpublished note \cite{pet4} in which he conjectures that whenever $s_{a+b+4}>0$, the Gysin map is injective.   We will prove the following partial analogue of Theorem \ref{gysin map 1} by adopting Petersen's approach.
\begin{theorem}\label{gysin2} Assume that $a+b$ is even and $a>b$. If $s_{a+b+4}>0$, the restriction of the Gysin map $\delta:H^1(\D_{1,1}, \V_{a+b}(-1))\to H^3(\A_2, \V_{a+b})$ to the $\Q$-sub HS $\Q(-a-2)$ of $H^1(\D_{1,1}, \V_{a+b}(-1))$ is nontrivial. 	
\end{theorem}

\begin{corollary} \label{relations} If $s_{a+b+4}>0$, then $(H^2(\u_2)\otimes V_{a+b})^{\Sp_2}=0$. Furthermore, $H_2(\u_2)$ is an $\Sp_2$-subrepresentation of the finite direct sum $\bigoplus_{a+b}V_{a+b}(a+2)$, where the sum is taken over the partitions $a+b$ of 2, 4, 6, 10 with $a>b$. 
\end{corollary}
\begin{remark}
	The weights of $H_2(\u_2)$ are given by $-a+b-4$ for $a>b$, and so we see that there are no quadratic relations in $\Gr^W_\bullet\u_2$. The range of the degrees of the relations on $\Gr^W_\bullet\u_2$ is from the cubic to the septic. Thus $\Gr^W_\bullet\u_2$ is finitely presented. However, the author is not able to determine  $H_2(\u_2)$ explicitly at this time. Since $T_2$ is free, Hain's approach does not seem to extend to the case when $g=2$. 
\end{remark}

\subsection{An overview of Petersen's approach on the nonvanishing of Gysin maps} We will briefly go over Petersen's approach on the nonvanishing theorem \cite[\S 5, Thm.~5.1]{pet3}.
 Petersen's idea is that since the cohomology classes of one's interest are represented by Eisenstein series, one should be able to reduce the problem to computations at the boundary. Denote the Baily-Borel compactification of $\A_g$ by $\overline{\A}_g$. Let $j:\A_2\to\overline{\A}_2$ and $\tilde{j}:\A_1\times\A_1\to \overline{\A}_1\times\overline{\A}_1$ be the inclusions into their Baily-Borel compactifications. We will consider the Gysin map between the stalks of the pushforwards  $(\mathrm{R}\tilde{j}_\ast \V_{a+b})^{S_2}$ and $\mathrm{R}j_\ast \V_{a+b}$.  Let $i_0:\A_0\to \overline{\A}_2$ and $\tilde{i}_0:\A_0\times\A_0\to \overline{\A}_1\times\overline{\A}_1$ be the inclusions of their zero-dimensional strata into the respective Baily-Borel compactifications. There is a commutative diagram:
 $$\xymatrix{
 	H^{k-2}(\D_{1,1}, \V_{a+b}(-1))\ar[r]\ar[d]&(\tilde{i}_0^\ast\mathrm{R}^{k-2}\tilde{j}_\ast \V_{a+b}(-1))^{S_2}\ar[d]\\
 	H^{k}(\A_2, \V_{a+b})\ar[r]&i_0^\ast\mathrm{R}^kj_\ast\V_{a+b}
 }
 $$
The nonvanishing of the left-hand Gysin map will be shown if the right-hand Gysin map at stalks maps the class of the Eisenstein series of our interest to a nonzero class in $i_0^\ast\mathrm{R}^kj_\ast\V_{a+b}$.  \\
\indent The first key point  is  that  the stalks on the right  are expressed as the cohomology groups of reductive groups with the Lie algebra cohomology of nilpotent  Lie algebras as coefficient groups. Petersen calls this as Harder's formula (see \cite{LR} for more details). Furthermore, these stalks carry natural MHS. This fact does not play a role for $k=2$, but we will need this for $k=3$.  The zero-dimensional strata of $\AAbar$ and $\Atwobar$ correspond to  the parabolic subgroup $P=B\times B$ of $\SL_2(\Q)\times\SL_2(\Q)$ and the Siegel parabolic subgroup $Q$ of $\Sp_2(\Q)$, respectively, where $B$ is the Borel subgroup of $\SL_2(\Q)$.  The natural map $\A_1\times\A_1\to\A_2$ corresponds to an inclusion $\SL_2\times\SL_2\to \Sp_2$, so we may consider $\SL_2\times\SL_2$ as a subgroup of $\Sp_2$. Then $P$ is the intersection $Q\cap\SL_2\times\SL_2$ in $\Sp_2$. Let $M_P$ and $M_Q$ be the reductive quotients of $P$ and $Q$ by the unipotent radicals $N_P$ and $N_Q$, respectively. Denote the Lie algebras of $N_P$ and $N_Q$ by $\n_P$ and $\n_Q$, respectively. The idea behind the Harder's formula is that the cohomology of the sufficiently small enough open punctured neighborhoods of the cusps $i_0$ and $\tilde{i}_0$ with the coefficient local system $\V_{a+b}$ can be computed by identifying the neighborhoods as the fibrations over the locally symmetric spaces associated to $M_P$ and $M_Q$ whose fibers are the nilmanifolds associated to $N_P$ and $N_Q$, respectively. The Leray spectral sequences of the fibrations degenerate, and hence the local cohomology we are after are given by  $H^{\bullet}(M_{\ast}(\Z), H^{\bullet}(\n_\ast,V_{a+b}))$, where $\ast\in\{P,Q\}$. The Gysin maps on the stalks are obtained by considering the Leray spectral sequences associated to the two fibrations and the Gysin sequence for the embedding of the neighborhood of $\tilde{i}_0$ into that of $i_0$. This is nicely explained in \cite[\S 5.3]{pet3}. \\
\indent The second important point of Petersen's method is that the nonvanishing of the Gysin maps on the stalks is related to the nontriviality of the modular symbols associated to eigen cusp forms. It comes down to compute the restriction map $H^1_c(\A_1, \H_{m}\otimes \C)\to H^1_c(\mathbb{R}_{>0}, \H_{m}\otimes \C)=H_m\otimes\C$, which is induced by the inclusion of symmetric spaces $\mathbb{R}_{>0}\to \h$. The $\Gm$ action decomposes $H_m\otimes \C$ into $m+1$ one-dimensional eigenspaces $H_{m}\otimes \C=\bigoplus_{l=0}^{m}E_{m-2l}.$ For a cusp form $f$ of weight $m+2$, we have a class $[f]$ in $H^1_c(\A_1,\H_{m}\otimes \C)$ and we are interested in the image of $[f]$ in $\bigoplus_{l=0}^{m}E_{m-2l}$. For $k=2$,  the nonvanishing of the Gysin map at the stalks is equivalent to the nontriviality of the image $[f]$ in the eigenspace $E_0$, which is in fact related to the central value of the $L$-function attached to $f$, where $f$ here is a normalized Hecke eigenform. For $k=3$, we will consider the image of $[f]$ in $E_{-a+b-2}$, and the nontriviality will follow from the fact that $L(f,n+1)$ is nontrivial for 
$n\geq\frac{a+b+4}{2},$ which is a basic fact about the convergence of the corresponding Euler product. 
\subsection{Proof of Theorem \ref{gysin2}} As stated above, we will follow Petersen's approach with some modification to our case. 
We will use the same notation as above. The natural map $\A_1\times \A_1\to \A_2$ corresponds to the inclusion $\SL_2\times\SL_2\to \Sp_2$ given by
$$\begin{bmatrix}a&b\\c&d\end{bmatrix}\times \begin{bmatrix}a'&b'\\c'&d'\end{bmatrix}\mapsto
\begin{bmatrix}a&0&b&0\\
0&a'&0&b'\\
c&0&d&0\\
0&c'&0&d'
\end{bmatrix}.$$ 
The Siegel parabolic subgroup $Q$ of $\Sp_2$ consists of the matrices whose lower left quadrant is all zero. The parabolic subgroup $P$ of $\SL_2\times \SL_2$ is the intersection of $Q$ with $\SL_2\times \SL_2$ in $\Sp_2$. 
The reductive quotients $M_P$ and $M_Q$ of  $P$ and $Q$, respectively, consist of matrices of the form
$$\begin{bmatrix}a&0&0&0\\
                 0&a'&0&0\\
                 0&0&a^{-1}&0\\
                 0&0&0&a'^{-1}
                 \end{bmatrix},
                 \hspace{.5in}
                 \left[\begin{array}{cc}
                 A&0\\
                 0&A^{-T}
                 \end{array}\right],$$
                 respectively, where $A$ in $\GL_2$. 
       Therefore, we have $M_Q\cong \GL_2$. Let $B$ be the standard Borel subgroup of $\SL_2$ consisting of upper triangular matrices, and  $M_B$ the reductive quotient of $B$ by the unipotent radical $N_B$, whose elements are of the form 
       $$\begin{bmatrix}1&a\\0&1\end{bmatrix}.$$ Let $T$ be the maximal $\Q$-split torus of $\SL_2$. The isomorphism $\Gm\cong T$ is given by $$t\mapsto \begin{bmatrix}t&0\\0&t^{-1}\end{bmatrix}.$$ Thus the action of $T$ on $N_B$ has the weight $t^2$.     
        Denote the Lie algebra of $N_B$ by $\n_B$. Then $M_p=M_B\times M_B\cong \Gm\times\Gm$ and $\n_P=\n_B\times\n_B$. The group $\Gm(\Z)$ is cyclic of order 2, which we denote by $<\sigma>$. \\
\indent Assume that $a+b$ is even and $a>b$. The restriction of $\V_{a+b}$ to $\D_{1,1}$ contains a single copy of $\UU_{a-b,0}(-b)$ as a summand, and we have seen that $H^1(\D_{1,1}, \V_{a+b})=H^1(\D_{1,1}, \UU_{a-b,0}(-b))$. Thus, we will consider the following commutative diagram
$$\xymatrix{
	H^{1}(\D_{1,1}, \UU_{a-b,0}(-b-1))\ar[r]\ar[d]&\tilde{i}_0^\ast\mathrm{R}^{1}\tilde{j}_\ast \UU_{a-b,0}(-b-1)\ar[d]\\
	H^{3}(\A_2, \V_{a+b})\ar[r]&i_0^\ast\mathrm{R}^3j_\ast\V_{a+b}
}
$$
We have seen that $H^{1}(\D_{1,1}, \UU_{a-b,0}(-b-1))= W_{a+b+3} \oplus \Q(-a-2)$.  Harder's formula and a K\"{u}nneth formula give the isomorphisms 
\begin{align*}
\tilde{i}_0^\ast\mathrm{R}^{1}\tilde{j}_\ast \UU_{a-b,0}&\cong \bigoplus_{l=0}^1H^l(M_P(\Z), H^{1-l}(\n_P, U_{a-b,0}))\\
&\cong H^0(<\sigma>\times <\sigma>, H^1(\n_P, U_{a-b,0}))\\
&\cong H^0(<\sigma>\times <\sigma>, H^\bullet(\n_B, H_{a-b})\otimes H^\bullet(\n_B, \Q))\\
&\cong (H^1(\n_B, H_{a-b})\otimes H^0(\n_B, \Q))\oplus (H^0(\n_B, H_{a-b})\otimes H^1(\n_B, \Q)).
\end{align*}
Note that the action of $\sigma$ on the Lie algebra cohomology of $\n_B$ is trivial when $a+b$ is even. We will consider $H_{m}\otimes K$ as a $K$-vector space of homogeneous polynomials of degree $m$ in two variable $X$ and $Y$. In this paper, $K$ is either $\Q$ or $\C$. The action of $\SL_2(K)$ on $H_{m}\otimes K=K[X,Y]_m$ is given by 
$$\alpha\cdot P(X,Y)=\begin{bmatrix}
a&b\\c&d
\end{bmatrix}P(X,Y)=P(aX+cY,bX+dY).$$
Under this action of $\SL_2(\Q)$ on $H_m$, we have 
$$H^0(\n_B, H_m)\cong \Q X^m \text{ and } H^1(\n_B, H_m)\cong \Q Y^m\otimes \eta,$$
where the element $\eta$ is the dual of the generator of $\n_B$. Note that $T$ acts on $\eta$ by $t^{-2}$.  Thus  there is an isomorphism of MHSs
$$\tilde{i}_0^\ast\mathrm{R}^{1}\tilde{j}_\ast \UU_{a-b,0}\cong \Q Y^{a-b}\otimes \eta \oplus \Q X^{a-b}\otimes \eta,$$
where the left-hand term is pure of weight $2(a-b)+2$ and the right-hand term pure of weight $2$. 
\begin{proposition}
	The generator of the term $\Q(-(a-b)-1))$ in $H^{1}(\D_{1,1}, \UU_{a-b,0})$ maps nontrivially into the term $\Q Y^{a-b}\otimes \eta$. 
\end{proposition}
\begin{proof} Let $m=a-b$. Let $\Delta:\A_1\to \A_1\times \A_1$  and $\overline{\Delta}: \overline{\A}_1\to \AAbar$ be the diagonal maps. Consider the following commutative diagram:
	$$\xymatrix{
		H^1(\A_1, \H_{m})\ar[r]^r& H^1(\partial\A_1, \H_m)\\
		H^1(\A_1\times\A_1, \UU_{m,0})\ar[r]\ar[u]_{\Delta^\ast}& \tilde{i}_0^\ast\mathrm{R}^{1}\tilde{j}_\ast \UU_{m,0}\ar[u]_{\overline{\Delta}^\ast}
	}
	$$
	where $H^1(\partial\A_1, \H_m)$ is the boundary cohomology of $\A_1$ with the local coefficient system $\H_m$ and the upper horizontal map $r$ is the restriction map. It is well-known that the map $r$ takes the generator of $\Q(-m-1)\subset H^1(\A_1, \H_m)$ to the generator of $H^1(\partial\A_1, \H_m)\cong \Q Y^m\otimes \eta$. Thus the image of the term $\Q(-m-1)$ in the stalk $\tilde{i}_0^\ast\mathrm{R}^{1}\tilde{j}_\ast \UU_{m,0}$ is nontrivial. Furthermore, the weight of the term $\Q Y^{m}\otimes \eta$ is given by $m+m+2=2m+2$, which implies that the image of $\Q(-m-1)$ in $\tilde{i}_0^\ast\mathrm{R}^{1}\tilde{j}_\ast \UU_{m,0}$ is equal to $\Q Y^m\otimes \eta$. 

\end{proof}

Now, we will consider the Gysin map between the stalks. Harder's formula applied to the stalk $i^\ast_0\mathrm{R}^3j_\ast\V_{a+b}$ gives an isomorphism
$$i^\ast_0\mathrm{R}^3j_\ast\V_{a+b}\cong \bigoplus_{l=0}^3H^l(M_Q(\Z), H^{3-l}(\n_Q, V_{a+b})).$$

The coefficient group $H^\bullet(\n_Q, V_{a+b})$ are  $M_Q\cong \GL_2$ representations, and Kostants theorem \cite[Table 2.3.4]{sch} gives us the corresponding $M_Q$ representations. As in the case $k=2$ (Petersen's proof), we will only need to consider them as the representations of the derived subgroup of $M_Q$, which is isomorphic to $\SL_2$. These are 
$$H^l(\n_Q, V_{a+b})\cong	\left\{
\begin{array}{ll}
	H_{a-b} &\text{ for } l=0,3 \\
	H_{a+b+2}& \text{ for } l=1,2. \\

\end{array} 
\right. 
$$
Note that $H^\bullet(\GL_2(\Z), H_m)\cong H^\bullet(\SL_2(\Z), H_m)^{<\sigma>}$, and hence  the stalk above is isomorphic to $H^1(\SL_2(\Z), H_{a+b+2})^{<\sigma>} \cong H^1(\A_1, \H_{a+b+2})^{<\sigma>}$. \\
\indent The Gysin map $	\tilde{i}_0^\ast\mathrm{R}^{1}\tilde{j}_\ast \UU_{a-b,0}\to i_0^\ast\mathrm{R}^3j_\ast\V_{a+b}$ can be obtained by considering the embedding of the sufficiently small open neighborhood $W_P$ of $\tilde{i}_0$  in $\A_1\times\A_1$ into  that $W_Q$ of $i$ in $\A_2$. Let $\ast \in \{P,Q\}$. The open neighborhood $W_\ast$ can be identified as a fibration over the locally symmetric space associated to $M_\ast$ with the fiber $N_\ast(\Z)\backslash N_\ast(\mathbb{R})$. The embedding yields the commutative diagram:
$$\xymatrix{
	N_P(\Z)\backslash   N_P(\mathbb{R})\ar[r]\ar[d]&N_Q(\Z)\backslash  N_Q(\mathbb{R})\ar[d]\\
	W_P\ar[r]\ar[d]& W_Q\ar[d]\\
	M_P(\Z)\backslash M_P(\mathbb{R})/ K_{M_P}\ar[r]& M_Q(\Z)\backslash M_Q(\mathbb{R})/K_{M_Q}.
}
$$ 
 The fact that symmetric spaces $M_\ast(\mathbb{R})/K_{M_\ast}$ are contractible implies that the cohomology of $M_\ast(\Z)\backslash M_\ast(\mathbb{R})/ K_{M_\ast}$ agrees with the cohomology of $M_\ast(\Z)$. The van Est isomorphism describes the cohomology of $N_\ast(\Z)\backslash N_\ast(\mathbb{R})$ as that of the Lie algebra $\n_\ast$. Then the corresponding map between the cohomology groups $H^\bullet(M_\ast(\Z), H^\bullet(\n_\ast, \V))$ with $\ast\in\{P,Q\}$ arises as part of the Gysin map induced between  the Leray spectral sequences associated to these fibrations. Thus we obtain the commutative diagram 
$$\xymatrix{
	\tilde{i}_0^\ast\mathrm{R}^{1}\tilde{j}_\ast \UU_{a-b,0}\ar[d]^{\cong}\ar[r] & i_0^\ast\mathrm{R}^3j_\ast\V_{a+b}\ar[d]^{\cong}\\
	H^0(M_P(\Z), H^1(\n_P, U_{a-b,0}))\ar[r]^g & H^1(M_Q(\Z), H^2(\n_Q, V_{a+b}))\\
}
$$
Since the Hodge weights play no role for the rest of the proof, we will omit Tate twists. \\
\indent We are led to understand the bottom horizontal map $g$. 
Petersen's idea is that this map is induced by the restriction of a local system over $\A_1$ along the inclusion of the locally symmetric space $\mathbb{R}_{>0}\times B<\sigma>$ of $\Gm$ into $\A_1$ of $\SL_2$, where $B<\sigma>$ is the classifying space for the group $<\sigma>$, and that it can be considered as the evaluations of the modular symbols associated to normalized eigenforms of a certain weight. Consider the commutative diagram of groups
$$\xymatrix{
	1\ar[r]&\Gm\ar[r]\ar[d]&M_P\ar[r]^{\mathrm{det}_P}\ar[d]&\Gm\ar[r]\ar@{=}[d]&1\\
	1\ar[r]&\SL_2\ar[r]&M_Q\ar[r]^{\mathrm{det}_Q}&\Gm\ar[r]&1,
}
$$
where the map $\mathrm{det}_P$ associates $aa'$ to a matrix $\mathrm{diag}(a, a', a^{-1}, a'^{-1})$ in $M_P$ and the map $\mathrm{det}_Q$ $\mathrm{det}(A)$ to an matrix $A$ in $M_Q$, and the left-hand vertical map is induced by the middle inclusion. Assigning to each group in the diagram the corresponding locally symmetric space, we obtain an inclusion of fibrations. It then follows from the Leray spectral sequences of the two fibrations that there are isomorphisms

$$H^0(M_P(\Z), H^1(\n_P, U_{a-b,0}))\cong H^0(\mathbb{R}_{>0}, \Q Y^{a-b}\otimes\eta\oplus \Q X^{a-b}\otimes\eta)^{<\sigma>}$$
and 
$$H^1(M_Q(\Z), H^2(\n_Q, V_{a+b}))\cong H^1(\A_1, \H_{a+b+2})^{<\sigma>}.$$
For a rational representation $V$ of $M_Q$, there is the restriction 
$$H^\bullet_c(M_Q(\Z)\backslash M_Q(\mathbb{R})/K_{M_Q}, \V)\to H^\bullet_c(M_P(\Z)\backslash M_P(\mathbb{R})/K_{M_P}, \V),$$ 
where $\V$ is the local system corresponding to  $V$.  Then for $V=H^1(\n_Q, V_{a+b})$, we observe that the dual of the restriction contains the map $g$. Thus there is the commutative diagram
$$\xymatrix{
	H^0(M_P(\Z), H^1(\n_P, U_{a-b,0}))\ar[r]^{\cong\hspace{.3in}}\ar[d]_g&H^0(\mathbb{R}_{>0}, \Q Y^{a-b}\otimes\eta\oplus \Q X^{a-b}\otimes\eta)^{<\sigma>}\ar[d]\\
	H^1(M_Q(\Z), H^2(\n_Q, V_{a+b}))\ar[r]^{\cong}&H^1(\A_1, \H_{a+b+2})^{<\sigma>}.
}
$$
We also denote by $g$ the right-hand vertical map of the diagram. The map $g$ is the restriction to the $<\sigma>$- invariants of the map
$$\tilde{g}:H^0(\mathbb{R}_{>0}, \Q Y^{a-b}\otimes\eta\oplus \Q X^{a-b}\otimes\eta)\to H^1(\A_1, \H_{a+b+2}),$$
which is induced by the inclusion $\Gm\to \SL_2$.
We have seen that 
$$H^0(\mathbb{R}_{>0}, \Q Y^{a-b}\otimes\eta\oplus \Q X^{a-b}\otimes\eta)^{<\sigma>}\cong \Q Y^{a-b}\otimes \eta\oplus \Q X^{a-b}\otimes \eta,$$
and so $\tilde{g}$ factors as 
$$\Q Y^{a-b}\otimes \eta\oplus \Q X^{a-b}\otimes \eta\to H^1(\A_1, \H_{a+b+2})^{<\sigma>}\to H^1(\A_1, \H_{a+b+2}).$$
Recall that $s_{a+b+4}$ denotes the dimension of the space of the cusp forms of weight $a+b+4$ for $\SL_2(\Z)$.
\begin{proposition}\label{key fact 2}
	If $s_{a+b+4}>0$, then the generator $Y^{a-b}\otimes \eta$ maps nontrivially into $H^1(\A_1, \H_{a+b+2})$. 
\end{proposition}

Before getting to the proof of this claim, we will briefly review the periods and modular symbols of a cusp form, and the $L$-function attached to a Heck eigenform. Let $f$ be a cusp form of weight $m+2$. The {\it $n$th period} of $f$ is defined to be
$$r_n(f)=\int_0^{i\infty} f(z)z^n~dz$$

for $0\leq n\leq m$. By identifying $H_m\otimes \C=\C[X,Y]_m$, the {\it period polynomial} $r(f)\in H_m\otimes \C$ of $f$ is defined to be
$$r(f)(X,Y)=\int_0^{i\infty}f(z)(zX+Y)^m~dz=\sum_{n=0}^m \binom mn r_n(f) X^nY^{m-n}.$$ The differential form $f(z)(zX+Y)^mdz$ defines a class $[f]$ in $H^1_c(\A_1, \H_m)$ and integrating along the imaginary axis corresponds to the restriction map
$$H^1_c(\A_1, \H_m\otimes\C)\to H^1_c(\mathbb{R}_{>0}, \H_m\otimes \C)=H_m\otimes \C,$$
where $\mathbb{R}_{>0}$ is embedded into $\h$ by $a\mapsto ia$. Thus the class $[f]$ maps to the period polynomial $r(f)(
X,Y)$. The $\Gm$-action decomposes $H_m$ into $m+1$ one-dimensional eigenspaces
$$H_m= \bigoplus_{n=0}^{m}E_{-m+2n}.$$
We may identify $X^{n}Y^{m-n}$ as a generator of $E_{-m+2n}$, and so $E_{-m+2n}=\Q X^{n}Y^{m-n}$. Composing the restriction map with the projection 
$$p_n: H^1_c(\mathbb{R}_{>0}, \H_m\otimes \C)\to H^1_c(\mathbb{R}_{>0}, E_{-m+2n}\otimes\C)=\C X^nY^{m-n}=\C$$ 
gives an evaluation, called {\it a modular symbol}, associating the value $\binom mn r_n(f)$ to a cusp form $f$.  \\
Let $q=e^{2\pi i z}$ for $z\in \h$. Let $f(q)=\sum a(n) q^n$ be the expansion of $f$ on the $q$-disk.  If $f$ is a Hecke eigenform and normalized, i.e., $a(1)=1$, then the $L$-function of $f$ is defined to be the Dirichlet series
$$L(f,s)=\sum_{n>0}a(n)n^{-s},$$
which is known to be convergent for $\mathrm{Re}(s)> \frac{m}{2}+2$. 

Furthermore, as $f$ is a normalized eigenform, the $L$-function has an Euler product expansion

$$L(f,s)=\prod_p(1-a(p)p^{-s}+p^{k-1-2s})^{-1}. $$
 
In our case, this is a key point and  for $s=n+1$ with $2n> m+2$, the Euler product converges and hence $L(f, n+1)$ is nontrivial. A key relation between the $L$-function and the period of a normalized eigenform $f$ is given by the Mellin transform
$$L(f,s)=\frac{(2\pi)^s}{\G(s)}\int_0^{i\infty}(-iz)^sf(z)\frac{dz}{z}.$$
For $s=n+1$, we have
$$r_n(f)=n!(-2\pi i)^{-n-1}L(f,n+1).$$
Another key property of the $L$-function $L(f,s)$ is that there is a functional equation relating $L(f,s)$ and $L(f,m+2-s)$. 
Now, we will prove Proposition \ref{key fact 2}.
\begin{proof}
We note that the map $H^0(\mathbb{R}_{>0}, \Q Y^{a-b}\otimes \eta)\to H^1(\A_1, \H_{a+b+2})$ is the dual of the restriction map
$$H^1_c(\A_1, \H_{a+b+2})\to H^1_c(\mathbb{R}_{>0}, \Q Y^{a-b}\otimes \eta).$$
Here we are identifying $H_m\cong H_m^\ast$ induced by the polarization, and moreover there is an isomorphism $\Q Y^{a-b}\otimes \eta\cong E_{-(a-b)-2}$ as $\Gm$-representations.
Let $f$ be a normalized eigenform of weight $a+b+4$. Let $[f]$ be the class in $H^1_c(\A_1, \H_{a+b+2}\otimes\C)$ corresponding to the 1-form $f(z)(zX+Y)^{a+b+2}dz$. Then the image of the class $[f]$ in 
$$H^1_c(\mathbb{R}_{>0}, \C Y^{a-b}\otimes\eta)\cong E_{-(a-b)-2}\otimes \C\cong\C X^bY^{a+2}$$
is given by $\binom{a+b+2}{b} r_b(f) X^bY^{a+2}.$ Thus we consider the value $L(f,b+1)$. By the functional equation mentioned above, the value $L(f,b+1)$ is nontrivial if and only  if $L(f, a+3)$. The assumption that $a>b$ implies that $2(a+2)>a+b+4$, and hence the Euler product for $L(f,a+3)$ converges and so $L(f,a+3)$ is nontrivial. 
\end{proof}
This completes the proof of Theorem \ref{gysin2}.



\begin{thebibliography}{}
	\bibitem{Bor}
	A.~Borel:
	{\em Stable real cohomology of arithmetic groups II,}, Manifolds and Groups, Papers in Honor of Yozo Matsushima, Progress in Mathematics 14, Birkhauser, Boston, 1981, 21-55. MR 83h:22023
     
     \bibitem{FH}
     W.~Fulton and J.~Harris:
     {\em Representation theory: A first course} Graduate Texts in Mathematics, 129, Springer-Verlag, 1991.
     
      \bibitem{joh1}
      D.~Johnson:
      {\em An abelian quotient of the mapping class group $\T_g$}, Math.~Ann. 249 (1980), 225-242. MR 87a:57005
      
      
    \bibitem{joh}
    D.~Johnson:
    {\em The structure of the Torelli group. I. A finite set of generators for $\T$} Ann.~ of Math. (2), 118(3):423-442, 1983.
    
    \bibitem{hain1} 
    R.~Hain:
    {\em Completions of mapping class groups and the cycle $C - C^{-}$}, Comtemporary Math.~150 (1993),~ 75--105.
    
    \bibitem{hain4}
    R.~Hain:
    {\em Genus 3 Mapping Class Groups are not Kähler}, Journal of Topology, vol. 8 no. 1 (2015), pp. 213-246, Oxford University Press, ISSN 1753-8416.
    
    \bibitem{hain5}
    R.~Hain:
    {\em The Hodge-de~Rham theory of modular groups}, in Recent Advances in Hodge Theory, arxiv [arXiv:1403.6443] (2015), Cambridge University Press 
    
    \bibitem{hain3}
    R.~Hain:
    {\em The Hodge De Rham theory of relative Malcev completion}, Annales Scientifiques de l'Ecole Normale Superieure, vol. 31 no. 1 (1998), pp. 47-92
    
    \bibitem{hain0}
    R.~Hain:
    {\em Infinitesimal presentations of Torelli groups}, J.~Amer.~Math.~Soc. 10 (1997), 597-651.
    
    \bibitem{hain2}
    R.~Hain:
    {Relative weight filtrations on completions of mapping class groups}, in Groups of Diffeomorphisms, Advanced Studies in Pure Mathematics, vol.~52 (2008), pp.~309-368, Mathematical Society of Japan.
    
    \bibitem{kab}
    A.~Kabanov:
    {\em The second cohomology of moduli spaces of Riemann surfaces with twisted coefficients}, Thesis, Duke University, 1995.
    
    \bibitem{LR}
    E.~Looijenga and M.~Rapoport:
    {\em Weights in the local cohomology of a Baily-Borel compactification} Complex geometry and Lie Theory (Sundance, UT, 1989). Vol.~53. Proc. Sympos. Pure Math. Amer. Math. Soc., Providence, RI, 223-260.
    
    \bibitem{mes}
    G.~Mess:
    {\em The Torelli groups for genus for genus 2 and 3 surfaces}, Topology 31 (1992), 775-790.
    
    \bibitem{mor}
    J.~Morgan:
    {\em The algebraic topology of smooth algebraic varieties}, Publ.~ Math.~IHES 48 (1978), 137-204.
    
    \bibitem{pet1}
    D.~Petersen:
    {\em Cohomology of local systems on loci of $d$-elliptic abelian surfaces}, Michigan Math. J. 62 (4), (2013), 705-720.
    
     \bibitem{pet2}
     D.~Petersen:
     {\em Cohomology of local systems on the moduli of principally polarized abelian surfaces}, Pacific J. Math. 275, (2015), 39-61.
     
    \bibitem{pet3}
    D.~Petersen:
    {\em Tautological rings of spaces of pointed genus two curves of compact type}, Compositio Mathematica, 152, (2016), 1398-1420.
    
   \bibitem{pet4}
   D.~Petersen:
   {\em On a conjectural formula for the cohomology of local systems on $M_2$}, unpublished. 
   
   \bibitem{sch}
   J.~Schwerwer:
   {\em On Euler products and residual Eisenstein cohomology classes for Siegel modular varieties}, Forum Math. 7 (1), 1-28.
   
   \bibitem{ser}
   J.~-P.~Serre:
   {\em A course in arithmetic}, GTM 7, Springer-Verlag, 1973.
\end{thebibliography}
\end{document}